\documentclass[12pt]{article}
\usepackage[latin1]{inputenc}
\usepackage[british]{babel}
\usepackage{cmap}
\usepackage{lmodern}

\usepackage{amssymb, amsmath, amsthm}
\usepackage[a4paper,top=25mm,bottom=25mm,left=25mm,right=25mm]{geometry}
\usepackage{ragged2e}

\usepackage{authblk} 
\usepackage{pifont}
\usepackage{graphicx}
\usepackage[usenames,dvipsnames,svgnames,table]{xcolor}
\usepackage[figuresright]{rotating}
\usepackage{xtab} 
\usepackage{longtable} 
\usepackage{multirow}
\usepackage{footnote}
\usepackage[stable]{footmisc}
\usepackage{chngpage} 
\usepackage{pdflscape} 
\usepackage[nottoc,notlot,notlof]{tocbibind} 

\usepackage{pgfplots}
\pgfplotsset{every tick label/.append style={font=\footnotesize}}
\pgfplotsset{compat=1.14}
\usepackage{setspace}

\usepackage{array}
\newcolumntype{K}[1]{>{\centering\arraybackslash$}p{#1}<{$}}

\usepackage{algorithm}
\usepackage{algpseudocode}

\makesavenoteenv{tabular}
\usepackage{tabularx}
\usepackage{booktabs}
\usepackage{threeparttable}
\usepackage[referable]{threeparttablex} 
\newcolumntype{R}{>{\raggedleft\arraybackslash}X}
\newcolumntype{L}{>{\raggedright\arraybackslash}X}
\newcolumntype{C}{>{\centering\arraybackslash}X}
\newcolumntype{A}{>{\columncolor{gray!25}}C}
\newcolumntype{a}{>{\columncolor{gray!25}}c}

\newlength{\tablen}

\usepackage{dcolumn} 
\newcolumntype{.}{D{.}{.}{-1}}

\usepackage{tikz}
\usetikzlibrary{arrows, calc, matrix, patterns, positioning, trees}
\usepackage[semicolon]{natbib}
\usepackage[hyphens]{url}
\usepackage{hyperref} 
\hypersetup{
  colorlinks   = true,    		
  urlcolor     = blue,    		
  linkcolor    = blue,    		
  citecolor    = ForestGreen	
}
\usepackage{microtype}
\usepackage[justification=centering]{caption} 

\usepackage[labelformat=simple]{subcaption}

\DeclareCaptionLabelFormat{parenthesis}{(#2)}
\makeatletter
\renewcommand\p@subfigure{\arabic{figure}.}
\makeatother

\DeclareCaptionLabelFormat{parenthesis}{(#2)}
\makeatletter
\renewcommand\p@subtable{\arabic{table}.}
\makeatother

\usepackage{enumitem}

\setlist[itemize]{leftmargin=2.5\parindent}
\setlist[enumerate]{leftmargin=2.5\parindent}

\def\addlegendimage{\csname pgfplots@addlegendimage\endcsname}

\theoremstyle{plain}

\newtheorem{theorem}{Theorem}

\theoremstyle{definition}

\newtheorem{definition}{Definition}[section]
\newtheorem{example}{Example}

\theoremstyle{remark}

\newtheorem{remark}{Remark}

\makeatletter
\let\@fnsymbol\@alph
\makeatother

\def\keywords{\vspace{.5em} 
{\noindent \textit{Keywords}: }}

\def\AMS{\vspace{.5em} 
{\noindent \textbf{\emph{MSC} class}: }}

\def\JEL{\vspace{.5em} 
{\noindent \textbf{\emph{JEL} classification number}: }}

\title{A lexicographically optimal completion for pairwise comparison matrices with missing entries}
\author{Kolos Csaba \'Agoston\thanks{~Email: \emph{kolos.agoston@uni-corvinus.hu} \newline Corvinus University of Budapest (BCE), Institute of Operations and Decision Sciences, Department of Operations Research and Actuarial Sciences, Budapest, Hungary}
$\qquad \qquad$
\href{https://sites.google.com/view/laszlocsato}{L\'aszl\'o Csat\'o}\thanks{~Corresponding author. Email: \emph{laszlo.csato@sztaki.hu} \newline
Institute for Computer Science and Control (SZTAKI), E\"otv\"os Lor\'and Research Network (ELKH), Laboratory on Engineering and Management Intelligence, Research Group of Operations Research and Decision Systems, Budapest, Hungary \newline
Corvinus University of Budapest (BCE), Institute of Operations and Decision Sciences, Department of Operations Research and Actuarial Sciences, Budapest, Hungary}} 
\date{\today}

\def\Dedication{
{\noindent
``\emph{One of the most appealing properties of the nucleolus, \\
as a solution concept, is its uniqueness.}''\footnote{~Source: \citet[p.~1164]{Schmeidler1969}.}
}

\flushright
\begin{small}
(David Schmeidler: \emph{The nucleolus of a characteristic function game})
\end{small}

\vspace{0.5cm} 
\justify }

\begin{document}

\maketitle
\thispagestyle{empty}
\Dedication

\begin{abstract}
\noindent
Estimating missing judgements is a key component in many multi-criteria decision making techniques, especially in the Analytic Hierarchy Process. Inspired by the Koczkodaj inconsistency index and a widely used solution concept of cooperative game theory called the nucleolus, the current study proposes a new algorithm for this purpose. In particular, the missing values are substituted by variables, and the inconsistency of the most inconsistent triad is reduced first, followed by the inconsistency of the second most inconsistent triad, and so on. The necessary and sufficient condition for the uniqueness of the suggested lexicographically optimal completion is proved to be a simple graph-theoretic notion: the undirected graph associated with the pairwise comparisons, where the edges represent the known elements, should be connected. Crucially, our method does not depend on an arbitrarily chosen measure of inconsistency as there exists essentially one reasonable triad inconsistency index.


\AMS{90-10, 90B50, 91B08}

\JEL{C44, D71}
\end{abstract}

\clearpage

\section{Introduction} \label{Sec1}

\subsubsection*{Pairwise comparisons and inconsistency}

In many decision-making problems, the importance of the $n$ alternatives should be measured numerically. Since asking each weight directly from an expert imposes a substantial cognitive burden, especially if $n$ is large, it is a common strategy to divide the problem into the simplest subproblems. Thus, the decision-maker is required to provide assessments on only two alternatives. However, the resulting pairwise comparisons are usually \emph{inconsistent}: even if alternative $i$ is two times better than alternative $j$ and alternative $j$ is three times better than alternative $k$, it is not guaranteed that alternative $i$ will be six times better than alternative $k$.
These deviations from a consistent set of pairwise comparisons can be quantified by inconsistency indices.

The inconsistency of complete pairwise comparisons is thoroughly studied in the literature and there exist dozens of different measures to determine the level of inconsistency \citep{BortotBrunelliFedrizziPereira2023, BozokiRapcsak2008, Brunelli2018, BrunelliFedrizzi2023, Cavallo2020}. One of the most popular indices has been suggested by Waldemar W.~Koczkodaj \citep{Koczkodaj1993, DuszakKoczkodaj1994}, which is the first with an axiomatic characterisation \citep{Csato2018a}. It is based on triads (pairwise comparisons among three alternatives) and identifies the inconsistency of a pairwise comparison matrix with the highest inconsistency of all its triads.

\subsubsection*{Estimating missing judgements}

In practice, some pairwise comparisons may be missing due to a lack of knowledge, time pressure, uncertainty, or other factors. Incomplete pairwise comparison matrices are often handled by filling them completely \citep{Harker1987a, Harker1987b}. An attractive way of completing a matrix is to formulate an optimisation problem where the missing entries are substituted by variables and the objective function is given by an inconsistency index of the corresponding complete matrix \citep{KoczkodajHermanOrlowski1999, TekileBrunelliFedrizzi2023, UrenaChiclanaMorente-MolineraHerrera-Viedma2015}. However, the implied non-linear optimisation problem may be difficult to solve.

The approach of minimising global inconsistency has been suggested for the consistency index of Saaty \citep{Saaty1977} in \citet{ShiraishiObataDaigo1998} and \citet{ShiraishiObata2002}. The resulting problem has been solved by \citet{BozokiFulopRonyai2010}. \citet{BozokiFulopRonyai2010} have also considered the logarithmic least squares objective function to achieve the best completion according to the geometric consistency index \citep{CrawfordWilliams1985, AguaronMoreno-Jimenez2003, AguaronEscobarMoreno-Jimenez2021}. Some authors have followed the same idea by finding a filling that minimises other inconsistency indices \citep{ErguKou2013, FedrizziGiove2007}.

On the other hand, \citet{SirajMikhailovKeane2012} propose a method based on the generation of all possible preferences (in the graph representation, all spanning trees) from the set of comparisons. \citet{BozokiTsyganok2019} prove the equivalence of the geometric mean of weight vectors calculated from all spanning trees and the logarithmic least squares problem.
\citet{PanLuLiuDeng2014} apply Dempster--Shafer evidence theory and information entropy to rank the alternatives by incomplete pairwise comparisons. \citet{ZhouHuDengChanIshizaka2018} develop a procedure to estimate missing values based on decision-making and trial evaluation laboratory (DEMATEL) method rather than minimising the level of inconsistency.

An inconsistency index is sometimes normalised by dividing it with a constant to get an (in)consistency ratio \citep{AgostonCsato2022, Saaty1977}. However, this transformation does not affect an algorithm that aims to obtain an optimal completion of missing judgements for a particular inconsistency index. 

\subsubsection*{Motivation and the main contributions}

Unexpectedly, the extant literature does not analyse how the completion of an incomplete pairwise comparison matrix can be optimised with respect to the well-established Koczkodaj inconsistency index. Probably, the reason is rather trivial: since this measure of inconsistency depends only on the most inconsistent triad, uniqueness is not guaranteed.

The current paper aims to fill that research gap. Therefore, we adopt the idea behind the nucleolus, a widely used solution concept in cooperative game theory \citep{Schmeidler1969}. To be more specific, the possible completions are ordered lexicographically based on the inconsistencies of all triads, and the one that gives the minimal element of this ordering is picked up. In other words, the inconsistency of the most inconsistent triad is reduced first, followed by the inconsistency of the second most inconsistent triad, and so on.

The main theoretical finding of our study resides in proving a simple graph-theoretic condition for the uniqueness of the proposed completion. In particular, the solution is unique if and only if the undirected graph where the nodes are the alternatives and the edges represent the known entries of the pairwise comparison matrix is connected. This is a natural necessary condition but sufficiency is non-trivial.

It is also presented how the optimal filling can be obtained by solving successive linear programming (LP) models.
Last but not least, we provide a simulation comparison of the suggested lexicographically optimal completion and the matrix that achieves the lowest possible inconsistency according to Saaty's consistency index.

\subsubsection*{The competitive edge of our proposal}

The major advantage of the lexicographically optimal completion method compared to other procedures suggested in the literature is that it essentially does not depend on an arbitrarily chosen measure of inconsistency: there is a unique reasonable index to quantify the inconsistency of triads \citep{Csato2019b}---which is the root cause of inconsistency in any pairwise comparison matrix---and lexicographic minimisation is a straightforward approach in optimisation. This is important because a plethora of inconsistency indices exist, each of them reflecting a different perspective on the quantification of inconsistency \citep{Brunelli2018}. However, among them, only the Koczkodaj inconsistency index does \emph{not} allow to compensate for the increased inconsistency of a particular triad with a reduction in the inconsistency of other triads. Therefore, the lexicographically optimal completion seems to be extremal among all optimisation methods based on inconsistency measures. 

Naturally, there are other approaches to fill in the missing elements \citep{Harker1987a, Harker1987b, TekileBrunelliFedrizzi2023, UrenaChiclanaMorente-MolineraHerrera-Viedma2015, ZhouHuDengChanIshizaka2018}. However, since inconsistency is strongly related to the validity of the priorities derived from the judgements \citep{Ozdemir2005}, ensuring a low level of inconsistency is crucial to obtain reliable results.

\subsubsection*{Practical relevance}

The proposed technique can be used to obtain the missing values in an incomplete pairwise comparison matrix, which is a key component in many multi-criteria decision making (MCDM) methods \citep{KouErguLinChen2016}, especially in the classical Analytic Hierarchy Process (AHP) \citep{IshizakaLabib2011, Saaty1977, Saaty1980}. Therefore, it might add value to any practical or managerial applications of the AHP \citep{BhushanRai2007, FormanGass2001, PereiraBamel2023, VaidyaKumar2006, Vargas1990} if
\begin{enumerate}[label=(\alph*)]
\item
the expert is not able to form a strong opinion on a particular judgement; or
\item
some pairwise comparisons have been lost; or
\item
the required number of questions ($n(n-1)/2$) is impossible to ask. 
\end{enumerate}
Finally, as a small example will show in Section~\ref{Sec6}, our algorithm provides a potentially different ranking in sports for players whose performances have been assessed by incomplete pairwise comparison matrices \citep{BozokiCsatoTemesi2016, ChaoKouLiPeng2018, Csato2013a, PetroczyCsato2021}. 

\subsubsection*{Structure}

The paper is organised as follows. Section~\ref{Sec2} presents basic definitions for incomplete pairwise comparison matrices. The lexicographically optimal completion is introduced and illustrated in Section~\ref{Sec3}. The central result, the necessary and sufficient condition for the uniqueness of the optimal solution is given in Section~\ref{Sec4}. The case of independent missing comparisons (placed in different rows) is studied in Section~\ref{Sec5}. Section~\ref{Sec6} analyses the similarity between the proposed technique and the completion minimising the inconsistency index of Saaty. A brief summary and discussions can be found in Section~\ref{Sec7}.

\section{Preliminaries} \label{Sec2}

The decision-maker is asked questions such as ``How many times alternative $i$ is preferred to alternative $j$?'', for which the numerical evaluation is $a_{ij}$. These relative judgements are collected into a matrix.
Let $\mathbb{R}_+$ denote the set of positive numbers, $\mathbb{R}^n_+$ denote the set of positive vectors of size $n$, and $\mathbb{R}^{n \times n}_+$ denote the set of positive square matrices of size $n$ with all elements greater than zero, respectively.

\begin{definition} \label{Def21}
\emph{Pairwise comparison matrix}:
Matrix $\mathbf{A} = \left[ a_{ij} \right] \in \mathbb{R}^{n \times n}_+$ is a \emph{pairwise comparison matrix} if $a_{ji} = 1 / a_{ij}$ for all $1 \leq i,j \leq n$.
\end{definition}

Denote the set of pairwise comparison matrices by $\mathcal{A}$ and the set of pairwise comparison matrices of order $n$ by $\mathcal{A}^{n \times n}$.
A pairwise comparison matrix of order three is called \emph{triad}.

In the ideal case, any indirect comparison of two alternatives leads to the same result as their direct comparison.

\begin{definition} \label{Def22}
\emph{Consistency}:
Pairwise comparison matrix $\mathbf{A} = \left[ a_{ij} \right] \in \mathcal{A}^{n \times n}$ is \emph{consistent} if $a_{ik} = a_{ij} a_{jk}$ holds for all $1 \leq i,j,k \leq n$.
\end{definition}

If the pairwise comparison matrix does not meet consistency, it is called \emph{inconsistent}. The degree of violating consistency can be measured by an inconsistency index.

\begin{definition} \label{Def23}
\emph{Inconsistency index}:
Let $\mathbf{A} \in \mathcal{A}$ be a pairwise comparison matrix.
\emph{Inconsistency index} $I : \mathcal{A} \to \mathbb{R}_+ \cup \{ 0 \}$ assigns a non-negative number to matrix $\mathbf{A}$.
\end{definition}

A widely used inconsistency index has been suggested by Waldemar W.~Koczkodaj \citep{Koczkodaj1993, DuszakKoczkodaj1994}. It focuses on the inconsistencies of individual triads and identifies the inconsistency of the whole matrix with the worst local inconsistency.

\begin{definition} \label{Def24}
\emph{Koczkodaj inconsistency index}:
Let $\mathbf{A} = \left[ a_{ij} \right] \in \mathcal{A}$ be an arbitrary pairwise comparison matrix.
Its \emph{Koczkodaj inconsistency index} $KI(\mathbf{A})$ is as follows:
\[
KI(\mathbf{A}) = \max_{1 \leq i,j,k \leq n} \min \left\{ \left| 1 - \frac{a_{ik}}{a_{ij} a_{jk}} \right|; \left| 1 - \frac{a_{ij} a_{jk}}{a_{ik}} \right| \right\}.
\]
\end{definition}

\begin{definition} \label{Def25}
\emph{Natural triad inconsistency index}:
Let $\mathbf{A} = \left[ a_{ij} \right] \in \mathcal{A}^{3 \times 3}$ be an arbitrary triad.
Its \emph{natural triad inconsistency index} $TI(\mathbf{A})$ is as follows:
\[
TI(\mathbf{A}) = \max \left\{ \frac{a_{ik}}{a_{ij} a_{jk}}; \frac{a_{ij} a_{jk}}{a_{ik}} \right\}.
\]
\end{definition}

According to equation (3.6) in \citet{BozokiRapcsak2008}, there exists a one-to-one correspondence between the Koczkodaj inconsistency index $KI$ and the natural triad inconsistency index $TI$ on the set of triads $\mathcal{A}^{3 \times 3}$:
\[
KI(\mathbf{A}) = 1- \frac{1}{TI(\mathbf{A})}.
\]

Sometimes, certain comparisons are missing, which will be denoted by $\ast$ in the following. 

\begin{definition} \label{Def26}
\emph{Incomplete pairwise comparison matrix}:
Matrix $\mathbf{A} = \left[ a_{ij} \right]$ is an \emph{incomplete pairwise comparison matrix} if $a_{ij} \in \mathbb{R}_+ \cup \{ \ast \}$ and for all $1 \leq i,j \leq n$, $a_{ij} \in \mathbb{R}_+$ implies $a_{ji} = 1 / a_{ij}$ and $a_{ij} = \ast$ implies $a_{ji} = \ast$.
\end{definition}

The set of incomplete pairwise comparison matrices of order $n$ is denoted by $\mathcal{A}_{\ast}^{n \times n}$. The number of unknown entries is denoted by $m$.

The structure of the known comparisons can be described by an undirected unweighted graph.

\begin{definition} \label{Def27}
\emph{Graph representation}:
The undirected graph $G = (V,E)$ represents the incomplete pairwise comparison matrix $\mathbf{A} \in \mathcal{A}_{\ast}^{n \times n}$ if
\begin{itemize}
\item
there is a one-to-one correspondence between the vertex set $V = \left\{ 1,2, \dots ,n \right\}$ and the alternatives;
\item
an edge is assigned to each known comparison outside the diagonal and vice versa, that is, $(i,j) \in E \iff a_{ij} \neq \ast$.
\end{itemize}
\end{definition}

\begin{example} \label{Examp1}
Consider the following incomplete pairwise comparison matrix of order four, where $a_{13}$ (thus $a_{31}$) and $a_{14}$ (thus $a_{41}$) remain undefined:
\[
\mathbf{A} = \left[
\begin{array}{K{3em} K{3em} K{3em} K{3em}}
    1     	& a_{12}  	& \ast   	& \ast   \\
    a_{21}	& 1       	& a_{23}	& a_{24} \\
   \ast		& a_{32}	& 1      	& a_{34} \\
   \ast 	& a_{42}  	& a_{43}	& 1 \\
\end{array}
\right].
\]

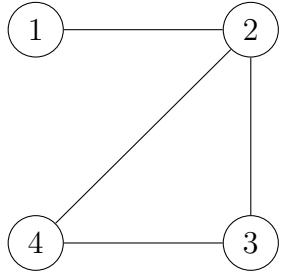
\begin{figure}[t!]
\centering
\begin{tikzpicture}[scale=1, auto=center, transform shape, >=triangle 45]
\tikzstyle{every node}=[draw,shape=circle];
  \node (n1) at (135:2) {$1$};
  \node (n2) at (45:2)  {$2$};
  \node (n3) at (315:2) {$3$};
  \node (n4) at (225:2) {$4$};

  \foreach \from/\to in {n1/n2,n2/n3,n2/n4,n3/n4}
    \draw (\from) -- (\to);
\end{tikzpicture}

\caption{The graph representation of the pairwise comparison matrix $\mathbf{A}$ in Example~\ref{Examp1}}
\label{Fig1}
\end{figure}

Figure~\ref{Fig1} shows the associated graph $G$.
\end{example}

\section{A lexicographically optimal completion of the missing elements} \label{Sec3}

Pairwise comparison matrices are used to determine a weight vector $\mathbf{w} = \left[ w_i \right]$ that adequately approximates the pairwise comparisons, namely, $w_i / w_j \approx a_{ij}$. There are several weighting methods for this purpose, see the surveys of \citet{GolanyKress1993} and \citet{ChooWedley2004}. However, they require all comparisons to be known, hence the missing elements should be reconstructed.

In our opinion, the most intuitively appealing technique for this purpose is to replace the $m$ missing comparisons with variables collected in a vector $\mathbf{x} \in \mathbb{R}_+^m$ and optimise a reasonable inconsistency index $I \left( A(\mathbf{x}) \right)$ of the implied complete pairwise comparison matrix $A(\mathbf{x})$ as the function of $\mathbf{x}$.
Saaty has proposed perhaps the most popular inconsistency measure $CR$ \citep{Saaty1977, Saaty1980}, which depends on the maximal eigenvalue corresponding to the complete pairwise comparison matrix. \citet{ShiraishiObataDaigo1998} and \citet{ShiraishiObata2002} aim to choose the missing values to minimise $CR$. The resulting optimisation problem has been solved by \citet{BozokiFulopRonyai2010}.

Another widely used inconsistency metric is the geometric inconsistency index $GCI$ \citep{CrawfordWilliams1985, AguaronMoreno-Jimenez2003, AguaronEscobarMoreno-Jimenez2021}, determined by the sum of squared errors $e_{ij} = \log a_{ij} - \log (w_i / w_j)$. Again, the associated optimisation problem has been solved by \citet{BozokiFulopRonyai2010}.

For both the $CR$- and $GCI$-optimal solutions, the necessary and sufficient condition for uniqueness is the connectedness of the representing graph. This is a natural requirement because two alternatives represented by vertices without a path between them cannot be evaluated on a common scale.

In contrast to most inconsistency indices including $CR$ and $GCI$, the Koczkodaj index $KI$ does not measure the average inconsistency of a pairwise comparison matrix but its highest local inconsistency. Consequently, the idea of a $KI$-optimal completion for an incomplete pairwise comparison matrix is worth investigating as it can provide an alternative way to reconstruct the missing elements.

\begin{example} \label{Examp2}
Consider the incomplete pairwise comparison matrix from Example~\ref{Examp1} such that $a_{12} = 2$, $a_{24} = 8$, and $a_{23} = a_{34} = 1$, while the missing elements are substituted by variables:
\[
\mathbf{A(\mathbf{x})} = \left[
\begin{array}{K{3em} K{3em} K{3em} K{3em}}
    1     	& 2			& x_{13}   	& x_{14} \\
    1/2		& 1       	& 1			& 8 \\
   1/x_{13}	& 1 		& 1      	& 1 \\
   1/x_{14}	& 1/8	 	& 1  		& 1 \\
\end{array}
\right].
\]
The matrix contains four triads with the following values of $TI$:
\[
TI_{123}(\mathbf{x}) = \max \left\{ \frac{x_{13}}{2}; \frac{2}{x_{13}} \right\};
\]
\[
TI_{124}(\mathbf{x}) = \max \left\{ \frac{x_{14}}{16}; \frac{16}{x_{14}} \right\};
\]
\[
TI_{134}(\mathbf{x}) = \max \left\{ \frac{x_{14}}{x_{13}}; \frac{x_{13}}{x_{14}} \right\};
\]
\[
TI_{234}(\mathbf{x}) = \max \left\{ 8; \frac{1}{8} \right\}.
\]

According to equation (3.6) in \citet{BozokiRapcsak2008}, the Koczkodaj inconsistency index of matrix $\mathbf{A(\mathbf{x})}$ is
\[
KI \left( \mathbf{A(\mathbf{x})} \right) = 1 - \frac{1}{\max \left\{ TI_{123}(\mathbf{x}); TI_{124}(\mathbf{x}); TI_{134}(\mathbf{x}); TI_{234}(\mathbf{x}) \right\}}.
\]
Consequently, $KI \left( \mathbf{A(\mathbf{x})} \right)$ is minimal if $t(\mathbf{x}) = \max \left\{ TI_{123}(\mathbf{x}); TI_{124}(\mathbf{x}); TI_{134}(\mathbf{x}); TI_{234}(\mathbf{x}) \right\}$ is minimal. Since $TI_{234}(\mathbf{x}) = 8$, $t(\mathbf{x}) \geq 8$. In particular, $t(\mathbf{x}) = 8$ if all of the following conditions hold:
\[
TI_{123}(\mathbf{x}) \leq 8 \iff 1/4 \leq x_{13} \leq 16;
\]
\[
TI_{124}(\mathbf{x}) \leq 8 \iff 2 \leq x_{14} \leq 128;
\]
\[
TI_{134}(\mathbf{x}) \leq 8 \iff 1/8 \leq x_{13} / x_{14} \leq 8;
\]
For example, $KI \left( \mathbf{A(\mathbf{x})} \right)$ is minimal (equals 8) if $x_{13} = 16$ and $2 \leq x_{14} \leq 128$.
\end{example}

Example~\ref{Examp2} shows that uniqueness cannot be guaranteed even if the representing graph is connected because the index $KI$ accounts for only the worst local inconsistency, thus substantial freedom remains for any other triad.

An analogous problem emerges in cooperative game theory when the attitude of the coalition that objects most strongly to a suggested payoff vector (the coalition with the greatest excess) should be reduced as low as possible. In the solution concept of the \emph{nucleolus} \citep{Schmeidler1969}, if there is an equality between the maximal excesses of two payoff vectors, the next greatest excesses are compared, and so on. In other words, the nucleolus is the lexicographically minimal vector of coalitional excesses.
This idea can immediately be adopted in our setting.

\begin{definition} \label{Def31}
\emph{Lexicographically optimal completion}:
Let $\mathbf{A} \in \mathcal{A}^\ast$ be an incomplete pairwise comparison matrix.
Let $\mathbf{A}(\mathbf{x})$ be the complete pairwise comparison matrix where the missing entries of matrix $\mathbf{A}$ are replaced by the variables collected in $\mathbf{x}$. Let $t_{ijk}(\mathbf{x})$ be the inconsistency of the triad determined by the three alternatives $1 \leq i,j,k \leq n$ according to the inconsistency index $TI$ in matrix $\mathbf{A}(\mathbf{x})$. 
Let $\theta(\mathbf{x})$ be the vector of the $n(n-1)(n-2)/6$ local inconsistencies $t_{ijk}(\mathbf{x})$ arranged in non-increasing order, that is, $\theta_u(\mathbf{x}) \geq \theta_v(\mathbf{x})$ for all $u < v$.

Matrix $\mathbf{A}(\mathbf{x})$ is said to be a \emph{lexicographically optimal completion} of the incomplete pairwise comparison matrix $\mathbf{A}$ if, for any other completion $\mathbf{A}(\mathbf{y})$, there does not exist an index $1 \leq v \leq n(n-1)(n-2)/6$ such that $\theta_u(\mathbf{x}) = \theta_u(\mathbf{y})$ for all $u < v$ and $\theta_v(\mathbf{x}) > \theta_v(\mathbf{y})$.
\end{definition}

\begin{example} \label{Examp3}
Consider the incomplete pairwise comparison matrix from Example~\ref{Examp2}.
The unique lexicographically optimal filling is $x_{13} = 4$ and $x_{14} = 8$ with $\theta(\mathbf{x}) = \left[ 8, 2, 2, 2 \right]$. If $x_{13} > 4$, then $t_{123}(\mathbf{x}) > 2$. If $x_{14} < 8$, then $t_{124}(\mathbf{x}) > 2$. Finally, if $x_{13} \leq 4$ and $x_{14} \geq 8$, then $t_{134}(\mathbf{x}) \geq 2$ and the last inequality is strict if either $x_{13} < 4$ or $x_{14} > 8$.
\end{example}

Similar to the nucleolus, the lexicographically optimal completion can be obtained by solving successive linear programming (LP) problems when one considers the logarithmically transformed entries of the original pairwise comparison matrix.

\begin{example} \label{Examp4}
Take the incomplete pairwise comparison matrix from Example~\ref{Examp1}.
The four triads imply eight constraints due to the reciprocity condition:
\begin{eqnarray*}
z_1 \to \min \\
\log a_{12} + \log a_{23} - \log x_{13} & \leq & z_1 \\
- \log a_{12} - \log a_{23} + \log x_{13} & \leq & z_1 \\
\log a_{12} + \log a_{24} - \log x_{14} & \leq & z_1 \\
- \log a_{12} - \log a_{24} + \log x_{14} & \leq & z_1 \\
\log x_{13} + \log a_{34} - \log x_{14} & \leq & z_1 \\
- \log x_{13} - \log a_{34} + \log x_{14} & \leq & z_1 \\
\log a_{23} - \log a_{24} + \log a_{34} & \leq & z_1 \\
- \log a_{23} + \log a_{24} - \log a_{34} & \leq & z_1
\end{eqnarray*}

As we have seen in Example~\ref{Examp2}, if $\log_2 a_{12} = 1$, $\log_2 a_{24} = 3$, and $\log_2 a_{23} = \log_2 a_{34} = 0$, then $z_1 = \bar{z_1} = 3$ due to the seventh constraint but there are multiple optimal solutions.
Another LP is obtained by removing the constraints associated with the triad $(2,3,4)$:
\begin{eqnarray*}
z_2 \to \min \\
\log a_{12} + \log a_{23} - \log x_{13} & \leq & z_2 \\
- \log a_{12} - \log a_{23} + \log x_{13} & \leq & z_2 \\
\log a_{12} + \log a_{24} - \log x_{14} & \leq & z_2 \\
- \log a_{12} - \log a_{24} + \log x_{14} & \leq & z_2 \\
\log x_{13} - \log a_{34} + \log x_{14} & \leq & z_2 \\
- \log x_{13} + \log a_{34} - \log x_{14} & \leq & z_2
\end{eqnarray*}
This problem already has a unique solution with $\log_2 x_{13} = 2$ ($x_{13} = 4$), $\log_2 x_{14} = 3$ ($x_{14} = 8$), and $z_2 = 1$.
\end{example}

To summarise, the lexicographically optimal completion can be obtained by an iterative algorithm as follows:
\begin{enumerate}
\item \label{LP_solution1}
A linear programming problem is solved to minimise the natural triad inconsistency index for all triads with an unknown value of $TI$.
\item
A triad (represented by two constraints in the LP), where the inconsistency index $TI$ cannot be lower, is chosen, which can be seen from the non-zero shadow price of at least one constraint.
\item
The inconsistency index $TI$ is fixed for this triad (or one of these triads if there exists more than one), the associated constraints are removed from the LP, and we return to Step~\ref{LP_solution1}.
\end{enumerate}
The process finishes if the minimal triad inconsistency indices are determined for all triads. The number of LPs to be solved is at most the number of triads having an incomplete pairwise comparison (which is finite) because the number of constraints in the LP decreases continuously.

In order to formulate this algorithm, some further notations are needed. Let $\mathcal{L}$ denote the index set of all triads. The elements of a triad $\ell \in \mathcal{L}$ are denoted by $i_\ell$, $j_\ell$, and $k_\ell$.

Consider the LP problem in the following form:
\begin{align*}
z & \rightarrow \min \tag{LP1.obj} \label{eq: obj_fun_LP} \\
\log a_{i_\ell,j_\ell} +  \log a_{j_\ell,k_\ell} + \log a_{k_\ell,i_\ell} &\le z_\ell  \qquad &\forall \ell\in\mathcal{L} \tag{LP1.1} \label{eq: g_elteres1_LP} \\
\log a_{i_\ell,j_\ell} +  \log a_{j_\ell,k_\ell} + \log a_{k_\ell,i_\ell} &\ge z_\ell  \qquad &\forall \ell\in\mathcal{L} \tag{LP1.2}\label{eq: g_elteres2_LP}
\\
z_\ell &\le z  \qquad &\forall \ell\in\mathcal{L} \tag{LP1.3}\label{eq: prem_LP} \\
\qquad \qquad  z_\ell &\ge 0  \qquad &\forall \ell\in\mathcal{L} \\
z &\ge 0\ \ ,\
\end{align*}
where $\log a_{i_\ell,j_\ell}$, $\log a_{j_\ell,k_\ell}$, $\log a_{k_\ell,i_\ell}$ is a parameter (unbounded decision variable) if the corresponding matrix element is known (missing). The suggested algorithm for the lexicographically optimal completion is provided in Algorithm~\ref{alg_LOC}.

\begin{algorithm}
\caption{Lexicographically optimal completion}
\label{alg_LOC}
\begin{algorithmic}[1]
\State \(\mathcal{M} \leftarrow \mathcal{L}\)
\State solve the LP problem LP1
\State \(obj \leftarrow\) objective value of LP1
\While{\(obj > 0\) }
\State find a constraint \(z_\ell\le z\) (\(\ell\in\mathcal{M}\)) for which the dual variable is negative
\State change constraint \(z_\ell\le z\) to \(z_\ell\le obj\)
\State \(\mathcal{M} \leftarrow \mathcal{M}\setminus\{\ell\}\)
\State solve the modified LP
\State \(obj \leftarrow\) objective value of the modified LP
\EndWhile
\end{algorithmic}

\end{algorithm}

Regarding the complexity of Algorithm~\ref{alg_LOC}, note that the number of variables in problem LP1 is at most $(n-1)(n-2)/2+n(n-1)(n-2)/6+1$: the number of missing elements in the pairwise comparison matrix plus the number of the triads (variables $z_\ell$) plus one (the variable $z$). The number of constraints is at most three times the number of triads, which is $n(n-1)(n-2)/2$. Since linear programming problems can be solved in polynomial time, and the number of the iterations in the ``while'' cycle of Algorithm~\ref{alg_LOC} is at most $n(n-1)(n-2)/2$ (the number of triads), Algorithm~\ref{alg_LOC} has a polynomial running time.
	
\begin{table}
\centering
\caption{Running times of Algorithm~\ref{alg_LOC}}
\label{Table1}
\rowcolors{1}{}{gray!20}
\begin{tabularx}{\textwidth}{CcC} \toprule
Matrix size ($n$) & Number of missing elements ($m$) & Running time (sec) \\ \bottomrule
    5     & 3     & 0.02 \\
    5     & 6     & 0.02 \\
    7     & 3     & 0.09 \\
    7     & 6     & 0.10 \\
    9     & 10    & 0.37 \\
    11    & 15    & 1.46 \\
    13    & 21    & 9.57 \\
    15    & 28    & 44.67 \\
    17    & 36    & 108.01 \\
    19    & 45    & 410.90 \\ \bottomrule
\end{tabularx}
\end{table}

Table~\ref{Table1} shows the running times for some particular pairs of matrix size $n$ and missing comparisons $m$, calculated as the average of ten randomly generated instances. Naturally, it is increasing with both variables. However, the running time is probably below one second for most problems arising in practice even if we do not make much effort to reduce it.

Finally, it is worth noting that the idea of going through the triads in decreasing order, starting from the most inconsistent one, appears in some inconsistency reduction methods for complete pairwise comparison matrices \citep{KoczkodajKosiekSzybowskiXu2015, KoczkodajSzybowski2016, MazurekPerzinaStrzalkaKowal2020}. However, while this problem has some connections to our setting, there is a crucial difference since we do not allow to change any known comparisons.

\section{A necessary and sufficient condition for the uniqueness of the lexicographically optimal completion} \label{Sec4}

In Examples~\ref{Examp3} and \ref{Examp4}, the lexicographic optimisation has resulted in a unique solution for the missing comparisons. However, it remains to be seen when this property holds in general.

In cooperative game theory, the nucleolus is always unique---but the proof is far from trivial, see \citet[Chapter~II.7]{Driessen1988}.
The necessary and sufficient condition for the uniqueness of our lexicographic optimisation problem is equivalent to the natural requirement that emerged in the case of the $CR$- and $GCI$-optimal completions: the graph representing the incomplete pairwise comparison matrix should be connected.

\begin{theorem} \label{Theo41}
The lexicographically optimal completion is unique if and only if the graph $G$ associated with the incomplete pairwise comparison matrix is connected.
\end{theorem}

\begin{proof}
\emph{Necessity}:
Suppose that $G$ is not connected. Then the incomplete pairwise comparison matrix $\mathbf{A} \in \mathcal{A}^\ast$ can be rearranged by an appropriate permutation of the alternatives into a decomposable form $\mathbf{D}$:
\[
\mathbf{D} = \left[
\begin{array}{K{2em} K{2em}}
    \mathbf{A_1}	& \mathbf{X} \\
    \mathbf{X}^\top	& \mathbf{A_2} \\
\end{array}
\right],
\]
where the square pairwise comparison (sub)matrices $\mathbf{A_1} \in \mathcal{A}^{p \times p}$ and $\mathbf{A_2} \in \mathcal{A}^{(n-p) \times (n-p)}$ contain all known elements (but they can also contain undefined elements). Consequently, $\mathbf{X}$ is composed of only missing elements.

Assume that $\mathbf{A}(\mathbf{x})$ is a lexicographically optimal completion for the incomplete pairwise comparison $\mathbf{A}$. Take its decomposable form $\mathbf{D}(\mathbf{x})$. $\mathbf{D}(\mathbf{y})$---thus, $\mathbf{A}(\mathbf{y})$---is shown to be a lexicographically optimal filling if $y_{ij} = \alpha x_{ij}$ for all $1 \leq i \leq p$ and $p+1 \leq j \leq n$, and $y_{ij} = x_{ij}$ otherwise.

Denote by $t_{ijk}(\mathbf{x})$ the natural triad inconsistency index $TI$ of any triad $1 \leq i,j,k \leq n$ in the matrix $\mathbf{A}(\mathbf{x})$, which may depend on the variables collected in $\mathbf{x}$.
$\mathbf{D}(\mathbf{y})$ consists of four types of triads:
\begin{itemize}
\item
if $1 \leq i,j,k \leq p$, then $t_{ijk}(\mathbf{y}) = t_{ijk}(\mathbf{x})$ since $y_{ij} = x_{ij}$, $y_{ik} = x_{ik}$, and $y_{jk} = x_{jk}$;
\item
if $1 \leq i,j \leq p$ and $p+1 \leq k \leq n$, then $t_{ijk}(\mathbf{y}) = t_{ijk}(\mathbf{x})$ since $y_{ij} = x_{ij}$, $y_{ik} = \alpha x_{ik}$, and $y_{jk} = \alpha x_{jk}$;
\item
if $1 \leq i \leq p$ and $p+1 \leq j,k \leq n$, then $t_{ijk}(\mathbf{y}) = t_{ijk}(\mathbf{x})$ since $y_{ij} = \alpha x_{ij}$, $y_{ik} = \alpha x_{ik}$, and $y_{jk} = x_{jk}$;
\item
if $p + 1 \leq i,j,k \leq n$, then $t_{ijk}(\mathbf{y}) = t_{ijk}(\mathbf{x})$ since $y_{ij} = x_{ij}$, $y_{ik} = x_{ik}$, and $y_{jk} = x_{jk}$.
\end{itemize}
To summarise, $\theta(\mathbf{y}) = \theta(\mathbf{x})$ but $\mathbf{y} \neq \mathbf{x}$ for any $\alpha \neq 1$.

\emph{Sufficiency}:
Assume that there are two different lexicographically optimal completions $\mathbf{A}(\mathbf{x}) \neq \mathbf{A}(\mathbf{y})$ and graph $G$ is connected.
\begin{enumerate}[label=\Roman*.]
\item
It is verified that $t_{ijk}(\mathbf{x}) = t_{ijk}(\mathbf{y})$ for all $1 \leq i,j,k \leq n$, namely, the inconsistency of any triad is the same for the two optimal fillings.

On the basis of the LP formulation presented in Section~\ref{Sec3}, it is clear that $t_{ijk} \left( (1- \varepsilon) \mathbf{x} + \varepsilon \mathbf{y} \right) \leq \max \{ t_{ijk} (\mathbf{x}); t_{ijk} (\mathbf{y}) \}$ for all $1 \leq i,j,k \leq n$ and for every $0 \leq \varepsilon \leq 1$. Furthermore, $t_{ijk} \left( (1- \varepsilon) \mathbf{x} + \varepsilon \mathbf{y} \right) < \max \{ t_{ijk} (\mathbf{x}); t_{ijk} (\mathbf{y}) \}$ if $t_{ijk} (\mathbf{x}) \neq t_{ijk} (\mathbf{y})$ and $0 < \varepsilon < 1$.
Consider the triad $i,j,k$ that has the highest inconsistency in matrix $\mathbf{A}(\mathbf{x})$ or $\mathbf{A}(\mathbf{y})$ with $t_{ijk}(\mathbf{x}) \neq t_{ijk}(\mathbf{y})$. It can be assumed without loss of generality that $t_{ijk}(\mathbf{x}) > t_{ijk}(\mathbf{y})$, when $t_{uvw}(\mathbf{x}) > t_{ijk}(\mathbf{x})$ implies $t_{uvw}(\mathbf{x}) = t_{uvw}(\mathbf{y})$ and $t_{fgh} (\mathbf{y}) > t_{fgh} (\mathbf{x})$ implies $t_{fgh} (\mathbf{y}) \leq t_{ijk} (\mathbf{x})$. Consequently, for any $0 < \varepsilon < 1$:
\begin{itemize}
\item
$t_{uvw}(\mathbf{x}) > t_{ijk}(\mathbf{x})$ implies $t_{uvw} \left( (1- \varepsilon) \mathbf{x} + \varepsilon \mathbf{y} \right) \leq t_{uvw} (\mathbf{x})$ as $t_{uvw} (\mathbf{x}) = t_{uvw} (\mathbf{y})$;
\item
$t_{ijk} \left( (1- \varepsilon) \mathbf{x} + \varepsilon \mathbf{y} \right) < t_{ijk} (\mathbf{x})$ as $t_{ijk} (\mathbf{x}) > t_{ijk} (\mathbf{y})$; and
\item
$t_{fgh}(\mathbf{x}) \leq t_{ijk}(\mathbf{x})$ implies $t_{fgh} \left( (1- \varepsilon) \mathbf{x} + \varepsilon \mathbf{y} \right) \leq \max \{ t_{fgh} (\mathbf{x}); t_{fgh} (\mathbf{y}) \} \leq t_{ijk} (\mathbf{x})$.
\end{itemize}
Then $\mathbf{x}$ is not an optimal completion because $(1-\varepsilon) \mathbf{x} + \varepsilon \mathbf{y}$ is lexicographically smaller.
To conclude, if two different lexicographically optimal completion would exist, a real convex combination of them would lead to a lower objective function value.

\item
The optimal fillings $\mathbf{x}$ and $\mathbf{y}$ differ, thus, there exist $1 \leq i,j \leq n$ such that $x_{ij} \neq y_{ij}$, therefore, $a_{ij}$ is missing.
Since graph $G$ is connected, there is a path $(i, k_1, k_2, \dots k_\ell, j)$ from node $i$ to node $j$. Owing to the previous part of the proof, $t_{ijk_\ell}(\mathbf{x}) = t_{ijk_\ell}(\mathbf{y})$, implying $x_{ik_\ell} \neq y_{ik_\ell}$ because $x_{ij} \neq y_{ij}$ and $a_{jk_\ell}$ is known. Consequently, $a_{ik_\ell}$ is a missing entry. The same argument can be used to show that $x_{ik_{\ell-1}} \neq y_{ik_{\ell-1}}$ and $a_{ik_{\ell-1}}$ is missing. Finally, the unknown element $a_{ik_2}$ is reached with $x_{ik_{2}} \neq y_{ik_{2}}$. However, this is the single missing entry in the triad $\left( i,k_1,k_2 \right)$, hence, $t_{i k_1 k_2}(\mathbf{x}) \neq t_{i k_1 k_2}(\mathbf{y})$, which contradicts the equation $t_{ijk}(\mathbf{x}) = t_{ijk}(\mathbf{y})$ derived before. 
\end{enumerate}
Obviously, the second part of verifying sufficiency relies on the connectedness of the associated graph $G$.
\end{proof}

\begin{remark}
The uniqueness of the lexicographically optimal solution depends only on the positions of comparisons (the structure of graph $G$) and is not influenced by the values of comparisons.
\end{remark}

\section{The case of independent missing entries} \label{Sec5}

Solving an LP model to obtain the lexicographically optimal completion may set back the use of the proposed method. Therefore, it is shown that the optimal completion can be calculated without an LP solver if the missing elements are independent, that is, they are placed in different rows and columns.

First, consider the case of one unknown comparison. Without losing generality, it can be assumed that $a_{1n} = x$ is missing. This element appears in $(n-2)$ triads because the third alternative cannot be the first and the last. The LP to be solved is
\begin{eqnarray*}
z_1 \to \min \\
\log a_{1k} + \log a_{kn} - \log x & \leq & z_1 \\
- \log a_{1k} - \log a_{kn} + \log x & \leq & z_1,
\end{eqnarray*}
where $2 \leq k \leq n-1$. Thus, there are $2(n-2)$ constraints such that the left hand side of each constraint is a linear function of the variable $x$. Therefore, the optimal value of $z_1$ is the minimum value of the upper envelope of $| \log a_{1k} + \log a_{kn} - \log x |$, $2 \leq k \leq n-1$.

Due to the lexicographic minimisation of triad inconsistencies, this idea can be used to obtain the optimal completion if all missing comparisons are independent since $\{ i,j \} \cap \{ k, \ell \} = \emptyset$ implies that no triad contains both $a_{ij}$ and $a_{k \ell}$.

\begin{example} \label{Examp5}
Consider the following incomplete pairwise comparison matrix of order five, where the independent entries $a_{15}$ (thus $a_{51}$) and $a_{24}$ (thus $a_{42}$) are undefined:
\[
\mathbf{A} = \left[
\begin{array}{K{3em} K{3em} K{3em} K{3em} K{3em}}
    1     	 & 1	 	&  1/6  &  1/4   & a_{15} \\
    1		 & 1     	&  1/9  & a_{24} & 1     \\
    6     	 & 9     	& 1     & 3     & 5     \\
    4     	 & 1/a_{24} &  1/3  & 1     & 1     \\
    1/a_{15} & 1     	&  1/5  & 1     & 1     \\
\end{array}
\right].
\]

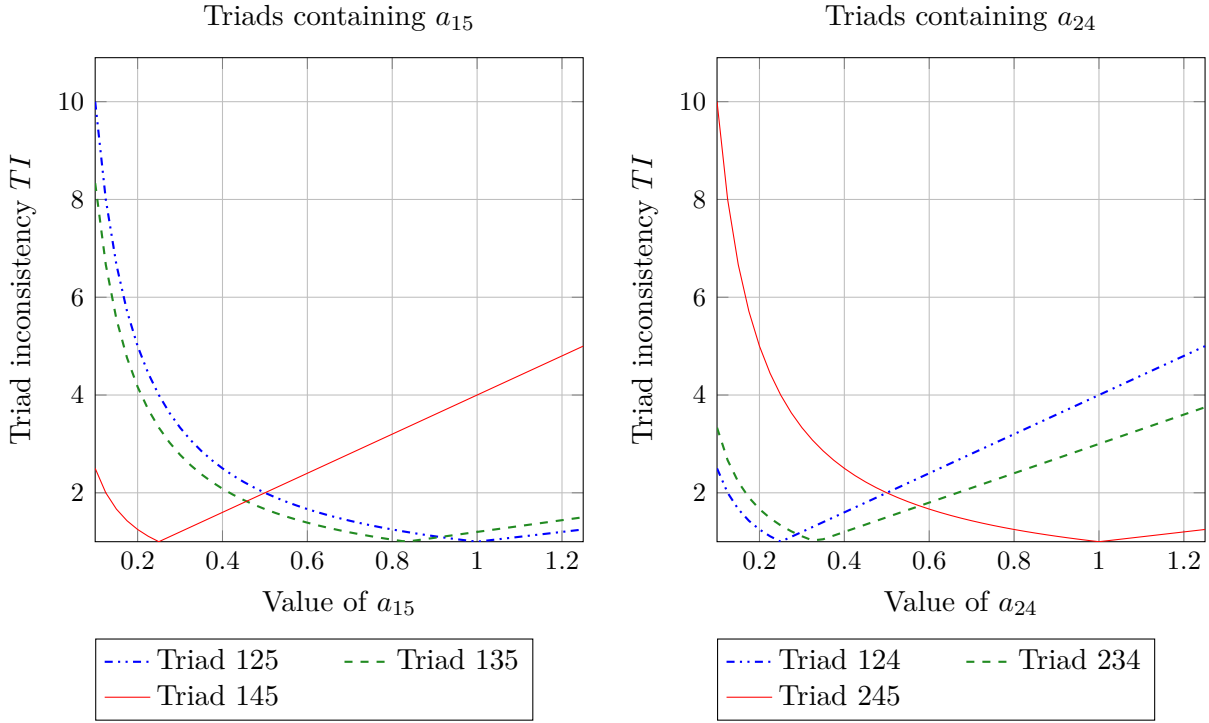
\begin{figure}[t!]
\centering

\begin{tikzpicture}
\begin{axis}[
name = axis1,
width = 0.5\textwidth, 
height = 0.5\textwidth,
title = {Triads containing $a_{15}$},
title style = {align=center, font=\small},
xmajorgrids = true,
ymajorgrids = true,
xmin = 0.1,
xmax = 1.25,
xlabel = {Value of $a_{15}$},
xlabel style = {align=center, font=\small},
scaled x ticks = false,
xticklabel style = {/pgf/number format/fixed,/pgf/number format/precision=2},
ymin = 1,
ylabel = {Triad inconsistency $TI$},
ylabel style = {align=center, font=\small},
legend style = {font=\small,at={(0,-0.2)},anchor=north west,legend columns=2},
legend entries = {Triad 125$\qquad$, Triad 135, Triad 145$\qquad$}
]
\addplot [blue, thick, dashdotdotted] coordinates{
(0.1,10)
(0.125,8)
(0.15,6.66666666666667)
(0.175,5.71428571428572)
(0.2,5)
(0.225,4.44444444444445)
(0.25,4)
(0.275,3.63636363636364)
(0.3,3.33333333333333)
(0.325,3.07692307692308)
(0.35,2.85714285714286)
(0.375,2.66666666666667)
(0.4,2.5)
(0.425,2.35294117647059)
(0.45,2.22222222222222)
(0.475,2.10526315789474)
(0.5,2)
(0.525,1.9047619047619)
(0.55,1.81818181818182)
(0.575,1.73913043478261)
(0.6,1.66666666666667)
(0.625,1.6)
(0.65,1.53846153846154)
(0.675,1.48148148148148)
(0.7,1.42857142857143)
(0.725,1.37931034482759)
(0.75,1.33333333333333)
(0.775,1.29032258064516)
(0.8,1.25)
(0.825,1.21212121212121)
(0.85,1.17647058823529)
(0.875,1.14285714285714)
(0.9,1.11111111111111)
(0.925,1.08108108108108)
(0.95,1.05263157894737)
(0.975,1.02564102564103)
(1,1)
(1.025,1.025)
(1.05,1.05)
(1.075,1.075)
(1.1,1.1)
(1.125,1.125)
(1.15,1.15)
(1.175,1.175)
(1.2,1.2)
(1.225,1.225)
(1.25,1.25)
};
\addplot [ForestGreen, thick, dashed] coordinates{
(0.1,8.333335)
(0.125,6.666668)
(0.15,5.55555666666667)
(0.175,4.76190571428572)
(0.2,4.1666675)
(0.225,3.70370444444445)
(0.25,3.333334)
(0.275,3.03030363636364)
(0.3,2.77777833333333)
(0.325,2.56410307692308)
(0.35,2.38095285714286)
(0.375,2.22222266666667)
(0.4,2.08333375)
(0.425,1.96078470588235)
(0.45,1.85185222222222)
(0.475,1.75438631578947)
(0.5,1.666667)
(0.525,1.58730190476191)
(0.55,1.51515181818182)
(0.575,1.44927565217391)
(0.6,1.38888916666667)
(0.625,1.3333336)
(0.65,1.28205153846154)
(0.675,1.23456814814815)
(0.7,1.19047642857143)
(0.725,1.14942551724138)
(0.75,1.11111133333333)
(0.775,1.07526903225806)
(0.8,1.041666875)
(0.825,1.01010121212121)
(0.85,1.01999979600004)
(0.875,1.04999979000004)
(0.9,1.07999978400004)
(0.925,1.10999977800004)
(0.95,1.13999977200005)
(0.975,1.16999976600005)
(1,1.19999976000005)
(1.025,1.22999975400005)
(1.05,1.25999974800005)
(1.075,1.28999974200005)
(1.1,1.31999973600005)
(1.125,1.34999973000005)
(1.15,1.37999972400005)
(1.175,1.40999971800006)
(1.2,1.43999971200006)
(1.225,1.46999970600006)
(1.25,1.49999970000006)
};
\addplot [red] coordinates{
(0.1,2.5)
(0.125,2)
(0.15,1.66666666666667)
(0.175,1.42857142857143)
(0.2,1.25)
(0.225,1.11111111111111)
(0.25,1)
(0.275,1.1)
(0.3,1.2)
(0.325,1.3)
(0.35,1.4)
(0.375,1.5)
(0.4,1.6)
(0.425,1.7)
(0.45,1.8)
(0.475,1.9)
(0.5,2)
(0.525,2.1)
(0.55,2.2)
(0.575,2.3)
(0.6,2.4)
(0.625,2.5)
(0.65,2.6)
(0.675,2.7)
(0.7,2.8)
(0.725,2.9)
(0.75,3)
(0.775,3.1)
(0.8,3.2)
(0.825,3.3)
(0.85,3.4)
(0.875,3.5)
(0.9,3.6)
(0.925,3.7)
(0.95,3.8)
(0.975,3.9)
(1,4)
(1.025,4.1)
(1.05,4.2)
(1.075,4.3)
(1.1,4.4)
(1.125,4.5)
(1.15,4.6)
(1.175,4.7)
(1.2,4.8)
(1.225,4.9)
(1.25,5)
};
\end{axis}

\begin{axis}[
at = {(axis1.south east)},
xshift = 0.11\textwidth,
width = 0.5\textwidth, 
height = 0.5\textwidth,
title = {Triads containing $a_{24}$},
title style = {align=center, font=\small},
xmajorgrids = true,
ymajorgrids = true,
xmin = 0.1,
xmax = 1.25,
xlabel = {Value of $a_{24}$},
xlabel style = {align=center, font=\small},
scaled x ticks = false,
xticklabel style = {/pgf/number format/fixed,/pgf/number format/precision=2},
ymin = 1,
ylabel = {Triad inconsistency $TI$},
ylabel style = {align=center, font=\small},
legend style = {font=\small,at={(0,-0.2)},anchor=north west,legend columns=2},
legend entries = {Triad 124$\qquad$, Triad 234, Triad 245$\qquad$}
]
\addplot [blue, thick, dashdotdotted] coordinates{
(0.1,2.5)
(0.125,2)
(0.15,1.66666666666667)
(0.175,1.42857142857143)
(0.2,1.25)
(0.225,1.11111111111111)
(0.25,1)
(0.275,1.1)
(0.3,1.2)
(0.325,1.3)
(0.35,1.4)
(0.375,1.5)
(0.4,1.6)
(0.425,1.7)
(0.45,1.8)
(0.475,1.9)
(0.5,2)
(0.525,2.1)
(0.55,2.2)
(0.575,2.3)
(0.6,2.4)
(0.625,2.5)
(0.65,2.6)
(0.675,2.7)
(0.7,2.8)
(0.725,2.9)
(0.75,3)
(0.775,3.1)
(0.8,3.2)
(0.825,3.3)
(0.85,3.4)
(0.875,3.5)
(0.9,3.6)
(0.925,3.7)
(0.95,3.8)
(0.975,3.9)
(1,4)
(1.025,4.1)
(1.05,4.2)
(1.075,4.3)
(1.1,4.4)
(1.125,4.5)
(1.15,4.6)
(1.175,4.7)
(1.2,4.8)
(1.225,4.9)
(1.25,5)
};
\addplot [ForestGreen, thick, dashed] coordinates{
(0.1,3.333333)
(0.125,2.6666664)
(0.15,2.222222)
(0.175,1.90476171428571)
(0.2,1.6666665)
(0.225,1.48148133333333)
(0.25,1.3333332)
(0.275,1.21212109090909)
(0.3,1.111111)
(0.325,1.02564092307692)
(0.35,1.05000010500001)
(0.375,1.12500011250001)
(0.4,1.20000012000001)
(0.425,1.27500012750001)
(0.45,1.35000013500001)
(0.475,1.42500014250001)
(0.5,1.50000015000002)
(0.525,1.57500015750002)
(0.55,1.65000016500002)
(0.575,1.72500017250002)
(0.6,1.80000018000002)
(0.625,1.87500018750002)
(0.65,1.95000019500002)
(0.675,2.02500020250002)
(0.7,2.10000021000002)
(0.725,2.17500021750002)
(0.75,2.25000022500002)
(0.775,2.32500023250002)
(0.8,2.40000024000002)
(0.825,2.47500024750003)
(0.85,2.55000025500003)
(0.875,2.62500026250003)
(0.9,2.70000027000003)
(0.925,2.77500027750003)
(0.95,2.85000028500003)
(0.975,2.92500029250003)
(1,3.00000030000003)
(1.025,3.07500030750003)
(1.05,3.15000031500003)
(1.075,3.22500032250003)
(1.1,3.30000033000003)
(1.125,3.37500033750003)
(1.15,3.45000034500003)
(1.175,3.52500035250004)
(1.2,3.60000036000004)
(1.225,3.67500036750004)
(1.25,3.75000037500004)
};
\addplot [red] coordinates{
(0.1,10)
(0.125,8)
(0.15,6.66666666666667)
(0.175,5.71428571428572)
(0.2,5)
(0.225,4.44444444444445)
(0.25,4)
(0.275,3.63636363636364)
(0.3,3.33333333333333)
(0.325,3.07692307692308)
(0.35,2.85714285714286)
(0.375,2.66666666666667)
(0.4,2.5)
(0.425,2.35294117647059)
(0.45,2.22222222222222)
(0.475,2.10526315789474)
(0.5,2)
(0.525,1.9047619047619)
(0.55,1.81818181818182)
(0.575,1.73913043478261)
(0.6,1.66666666666667)
(0.625,1.6)
(0.65,1.53846153846154)
(0.675,1.48148148148148)
(0.7,1.42857142857143)
(0.725,1.37931034482759)
(0.75,1.33333333333333)
(0.775,1.29032258064516)
(0.8,1.25)
(0.825,1.21212121212121)
(0.85,1.17647058823529)
(0.875,1.14285714285714)
(0.9,1.11111111111111)
(0.925,1.08108108108108)
(0.95,1.05263157894737)
(0.975,1.02564102564103)
(1,1)
(1.025,1.025)
(1.05,1.05)
(1.075,1.075)
(1.1,1.1)
(1.125,1.125)
(1.15,1.15)
(1.175,1.175)
(1.2,1.2)
(1.225,1.225)
(1.25,1.25)
};
\end{axis}
\end{tikzpicture}

\caption{Triad inconsistencies as a function of unknown entries in Example~\ref{Examp5}}
\label{Fig2}

\end{figure}


Both $a_{15}$ and $a_{24}$ are contained in three triads, respectively. The natural triad inconsistency index $TI$ of these triads depends on the value of the missing entries as shown in Figure~\ref{Fig2}. Hence, the inconsistencies of triads 125, 135, 145 are minimised lexicographically if $a_{15} = 0.5$ (where the minimum value of the upper envelope is reached on the left chart), and the inconsistencies of triads 124, 234, 245 are minimised lexicographically if $a_{24} = 0.5$ (where the minimum value of the upper envelope is reached on the right chart).
\end{example}

Even though the transformation above is not applicable if a row contains more than one unknown comparison, the number of variables in the LP remains $m$ (the number of missing elements) and the number of constraints is less than $2m(n-2)$. Consequently, the optimal completion can be calculated with essentially any software containing an LP solver module such as Python, R, and even Microsoft Excel (using the free Solver add-in).

\section{Numerical experiment: the similarity of completion methods} \label{Sec6}

If an incomplete pairwise comparison matrix can be filled out to get a consistent matrix, then all methods that optimise a reasonable inconsistency index of the complete matrix result in the same solution. However, this is not necessarily the case in general. The Koczkodaj inconsistency index---in contrast to most other measures of inconsistency---does not examine the average inconsistency of the set of pairwise comparisons but the highest local inconsistency. Therefore, the difference between the proposed lexicographically optimal completion and the filling which minimises the consistency index suggested by Saaty \citep{BozokiFulopRonyai2010, ShiraishiObataDaigo1998, ShiraishiObata2002} is worth exploring. As usual, the relationship is studied with respect to the inconsistency of the original incomplete pairwise comparison matrix.
Note that the $CR$-optimal completion has recently been compared to other techniques used to estimate the missing values \citep{TekileBrunelliFedrizzi2023}.

The similarity of the two procedures is examined by comparing the corresponding complete pairwise comparison matrices derived by minimising the given objective function. For this purpose, the incompatibility index defined by \citet{Saaty2008} and used in \citet{KulakowskiMazurekStrada2022} is adopted after some monotonic transformations.\footnote{~Previous authors have called it compatibility index, however, it measures the level of incompatibility.}

\begin{definition} \label{Def61}
\emph{Incompatibility index}:
Let $\mathbf{A} = \left[ a_{ij} \right] \in \mathcal{A}^{n \times n}$ and $\mathbf{B} = \left[ b_{ij} \right] \in \mathcal{A}^{n \times n}$ be two pairwise comparison matrices of the same order $n$. Their \emph{incompatibility index} is as follows:
\[
ICI \left( \mathbf{A}, \mathbf{B} \right) = 100 \times \left( \frac{1}{n^2} \sum_{i=1}^n \sum_{j=1}^n a_{ij} b_{ji} - 1 \right).
\]
\end{definition}
Clearly, $ICI \left( \mathbf{A}, \mathbf{B} \right) = 0$ if and only if matrices $\mathbf{A}$ and $\mathbf{B}$ coincide. Otherwise, this value is positive.

Four cases are investigated:
\begin{itemize}
\item
$n=5$ and $m=1$ (five alternatives with one missing comparison);
\item
$n=5$ and $m=2$ (five alternatives with two missing comparisons);
\item
$n=6$ and $m=6$ (six alternatives with six missing comparisons);
\item
$n=10$ and $m=1$ (ten alternatives with one missing comparison).
\end{itemize}

Our Monte Carlo experiment is implemented in the following way:
\begin{itemize}
\item
A random matrix is generated such that all randomly chosen known entries take a value from the Saaty scale of
\[
\{ 1/9,\, 1/8,\, \dots ,\, 1/2,\, 1,\, 2,\, \dots ,\, 9 \},
\]
independently from each other.

\item
The incomplete pairwise comparison matrix is retained if its (in)consistency ratio is below the $0.1$ threshold ($0.5$ if $n=10$ and $m=1$), which can be calculated using the values of the random index for incomplete pairwise comparison matrices provided in \citet[Table~2]{AgostonCsato2022}. However, for $n=10$ and $m=1$, the linear approximation of equation (3) in \citet{AgostonCsato2022} is considered due to the time-consuming computation.

\item
The connectedness of the associated graph is checked if $n=6$ and $m=6$ to guarantee that a unique optimal completion exists.
\end{itemize}
The procedure is repeated until we get 500 matrices satisfying the required properties.

Two remarks should be mentioned regarding the calculations. First, the lexicographically optimal completion requires the identification of binding constraints in LP problems. Therefore, we have fixed the right-hand side of one constraint with a positive dual variable in each step. Second, in order to avoid numerical difficulties in the eigenvalue minimisation problem, the pairwise comparison matrix is not considered if the optimal value of a missing entry is lower than $1/9$ or higher than $9$ . The latter issue has been discussed in more detail in \citet[Section~4]{AgostonCsato2022}.

\input{Figure3_optimal_completion_comparison}

The results are depicted in Figure~\ref{Fig3}. Unsurprisingly, the two completion methods are more similar if the number of alternatives $n$ or the number of missing elements $m$ is smaller. Analogously, the estimated values of the unknown entries are generally closer as the inconsistency of the incomplete matrix decreases.

Examples with a high incompatibility index below the 10\% threshold are worth further examination.
For $n=5$ and $m=1$, the lexicographical minimisation and the eigenvector method imply the following optimal solutions, respectively (the missing entries are written in bold):
\[
\mathbf{A}^{(1)} = \left[
\begin{array}{K{3.5em} K{3.5em} K{3.5em} K{3.5em} K{3.5em}}
    1     &  1/2  & 5     &  1/6  & \mathbf{0.2440} \\
    2     & 1     & 4     &  1/2  &  1/6 \\
     1/5  &  1/4  & 1     &  1/6  &  1/7 \\
    6     & 2     & 6     & 1     &  1/2 \\
    \mathbf{4.0988} & 6     & 7     & 2     & 1     \\
\end{array}
\right]
\text{ and}
\]
\[
\mathbf{A}^{(2)} = \left[
\begin{array}{K{3.5em} K{3.5em} K{3.5em} K{3.5em} K{3.5em}}
    1     &  1/2  & 5     &  1/6  & \mathbf{0.1798} \\
    2     & 1     & 4     &  1/2  &  1/6 \\
     1/5  &  1/4  & 1     &  1/6  &  1/7 \\
    6     & 2     & 6     & 1     &  1/2 \\
    \mathbf{5.5628} & 6     & 7     & 2     & 1     \\
\end{array}
\right].
\]
Here the unknown entry $a_{51}$ should be replaced with a value slightly above $4$ if the inconsistencies of the triads are minimised lexicographically but the corresponding variable is above $5.5$ if the dominant eigenvalue of the filled matrix is optimised.

For $n=5$ and $m=2$, the lexicographic minimisation, the eigenvector, and the logarithmic least squares methods \citep{BozokiFulopRonyai2010, BozokiTsyganok2019} provide the following optimal solutions, respectively (the missing entries are written in bold):
\[
\mathbf{B}^{(1)} = \left[
\begin{array}{K{3.5em} K{3.5em} K{3.5em} K{3.5em} K{3.5em}}
    1     & \mathbf{0.8274} &  1/6  &  1/4  & \mathbf{0.4564} \\
    \mathbf{1.2086} & 1     &  1/9  &  1/6  & 1     \\
    6     & 9     & 1     & 3     & 5     \\
    4     & 6     &  1/3  & 1     & 1     \\
    \mathbf{2.1909} & 1     &  1/5  & 1     & 1     \\
\end{array}
\right];
\]
\[
\mathbf{B}^{(2)} = \left[
\begin{array}{K{3.5em} K{3.5em} K{3.5em} K{3.5em} K{3.5em}}
    1     & \mathbf{1.0993} &  1/6  &  1/4  & \mathbf{0.6047} \\
    \mathbf{0.9097} & 1     &  1/9  &  1/6  & 1     \\
    6     & 9     & 1     & 3     & 5     \\
    4     & 6     &  1/3  & 1     & 1     \\
    \mathbf{1.6537} & 1     &  1/5  & 1     & 1     \\
\end{array}
\right];
\]
\[
\mathbf{B}^{(3)} = \left[
\begin{array}{K{3.5em} K{3.5em} K{3.5em} K{3.5em} K{3.5em}}
    1     & \mathbf{1.1141} &  1/6  &  1/4  & \mathbf{0.6146} \\
    \mathbf{0.8976} & 1     &  1/9  &  1/6  & 1     \\
    6     & 9     & 1     & 3     & 5     \\
    4     & 6     &  1/3  & 1     & 1     \\
    \mathbf{1.6272} & 1     &  1/5  & 1     & 1     \\
\end{array}
\right].
\]
Now the lexicographic optimisation means that $b^{(1)}_{12}$ is below one, however, the eigenvector and logarithmic least squares minimisations imply that $b^{(2)}_{12}$ and $b^{(3)}_{12}$ are above one. Furthermore, the row geometric mean ($GM$) and the eigenvector ($EM$) weighting methods result in the following priorities:
\[
\mathbf{w}^{(GM)} \left( \mathbf{B}^{(1)} \right) = \left[
\begin{array}{K{3em} K{3em} K{3em} K{3em} K{3em}}
6.153 & 6.602 & 53.879 & 21.396 & 11.969 \\
\end{array}
\right];
\]
\[
\mathbf{w}^{(EM)} \left( \mathbf{B}^{(1)} \right) = \left[
\begin{array}{K{3em} K{3em} K{3em} K{3em} K{3em}}
5.988 & 6.810 & 52.723 & 22.162 & 12.317 \\
\end{array}
\right];
\]
\[
\mathbf{w}^{(GM)} \left( \mathbf{B}^{(2)} \right) = \left[
\begin{array}{K{3em} K{3em} K{3em} K{3em} K{3em}}
6.909 & 6.255 & 54.032 & 21.457 & 11.346 \\
\end{array}
\right];
\]
\[
\mathbf{w}^{(EM)} \left( \mathbf{B}^{(2)} \right) = \left[
\begin{array}{K{3em} K{3em} K{3em} K{3em} K{3em}}
6.716 & 6.458 & 52.693 & 22.302 & 11.831 \\
\end{array}
\right];
\]
\[
\mathbf{w}^{(GM)} \left( \mathbf{B}^{(3)} \right) = \left[
\begin{array}{K{3em} K{3em} K{3em} K{3em} K{3em}}
6.951 & 6.239 & 54.039 & 21.460 & 11.311 \\
\end{array}
\right];
\]
\[
\mathbf{w}^{(EM)} \left( \mathbf{B}^{(3)} \right) = \left[
\begin{array}{K{3em} K{3em} K{3em} K{3em} K{3em}}
6.758 & 6.441 & 52.688 & 22.310 & 11.803 \\
\end{array}
\right].
\]
It can be seen that the first alternative has a smaller weight than the second according to both procedures for the first matrix $\mathbf{B}^{(1)}$, while the first alternative is judged better than the second according to both algorithms based on the second matrix $\mathbf{B}^{(2)}$ and the third matrix $\mathbf{B}^{(3)}$. 

Consequently, the lexicographically optimal completion method might lead to a different ranking of the alternatives than other widely used procedures even in the relatively simple case of five alternatives and two missing comparisons. This is far from surprising since the suggested algorithm focuses on the triads separately, while the eigenvector and logarithmic least squares techniques consider inconsistencies across the whole matrix. Nonetheless, these incomplete pairwise comparison matrices might require further investigation as they potentially imply an uncertain ranking of the alternatives.

\section{Conclusions} \label{Sec7}

We have proposed a new procedure to reconstruct the missing values in an incomplete pairwise comparison matrix. Inspired by the idea behind the nucleolus, a desirable payoff-sharing solution in cooperative game theory, our approach chooses the unknown comparisons such that the inconsistencies of the triads will be lexicographically minimal. Our main theoretical result is proving a natural sufficient and necessary condition for the uniqueness of this completion. In particular, the associated undirected graph should be connected, which can be easily checked and does not depend on the value of the known comparisons. Even though the same condition is required for the eigenvalue-based and the logarithmic least squares optimal fillings \citep{BozokiFulopRonyai2010}, our algorithm does not depend on a particular inconsistency index chosen from the plethora of available measures \citep{Brunelli2018} as there exists essentially one triad inconsistency index.
The lexicographically optimal solution has been compared to the eigenvalue minimisation technique through a numerical Monte Carlo experiment.

There are several interesting directions to continue this work. In cooperative game theory, computing the nucleolus is a challenging task \citep{BenedekFliegeNguyen2021}. Therefore, it might be non-trivial to obtain the lexicographically minimal solution for some incomplete matrices considered in the literature: \citet{Csato2013a} has analysed a matrix with 149 alternatives where the ratio of known comparisons is below 10\%.
Since the optimal completion with respect to the geometric consistency index has nice graph theoretical relations \citep{Bozoki2017, BozokiTsyganok2019, ChenKouMichaelTarnSong2015}, it remains to be seen whether our proposal can be supported by similar arguments. 
Other aspects of the problem can be taken into account to reproduce the missing entries. For instance, \citet{FaramondiOlivaBozoki2020} compute ordinal preferences first, then seek a cardinal ranking by enforcing these ordinal constraints. Finally, a thorough comparison of various completion methods proposed for pairwise comparison matrices with missing entries can support the decision-maker in dealing with unknown comparisons.

\section*{Acknowledgements}
\addcontentsline{toc}{section}{Acknowledgements}
\noindent
We are grateful to \emph{S\'andor Boz\'oki} and \emph{Zsombor Sz\'adoczki} for useful advice. \\
Eight anonymous reviewers provided valuable comments and suggestions on earlier drafts.

\bibliographystyle{apalike}
\bibliography{All_references}

\begin{thebibliography}{}

\bibitem[{\'A}goston and Csat{\'o}, 2022]{AgostonCsato2022}
{\'A}goston, K.~{\relax Cs}. and Csat{\'o}, L. (2022).
\newblock Inconsistency thresholds for incomplete pairwise comparison matrices.
\newblock {\em Omega}, 108:102576.

\bibitem[Aguar{\'o}n et~al., 2021]{AguaronEscobarMoreno-Jimenez2021}
Aguar{\'o}n, J., Escobar, M.~T., and Moreno-Jim{\'e}nez, J.~M. (2021).
\newblock Reducing inconsistency measured by the {G}eometric {C}onsistency
  {I}ndex in the {A}nalytic {H}ierarchy {P}rocess.
\newblock {\em European Journal of Operational Research}, 288(2):576--583.

\bibitem[Aguar\'on and Moreno-Jim{\'e}nez, 2003]{AguaronMoreno-Jimenez2003}
Aguar\'on, J. and Moreno-Jim{\'e}nez, J.~M. (2003).
\newblock The geometric consistency index: Approximated thresholds.
\newblock {\em European Journal of Operational Research}, 147(1):137--145.

\bibitem[Benedek et~al., 2021]{BenedekFliegeNguyen2021}
Benedek, M., Fliege, J., and Nguyen, T.-D. (2021).
\newblock Finding and verifying the nucleolus of cooperative games.
\newblock {\em Mathematical Programming}, 190(1):135--170.

\bibitem[Bhushan and Rai, 2007]{BhushanRai2007}
Bhushan, N. and Rai, K. (2007).
\newblock {\em Strategic Decision Making: Applying the Analytic Hierarchy
  Process}.
\newblock Springer Science \& Business Media, London.

\bibitem[Bortot et~al., 2023]{BortotBrunelliFedrizziPereira2023}
Bortot, S., Brunelli, M., Fedrizzi, M., and Pereira, R.~A.~M. (2023).
\newblock A novel perspective on the inconsistency indices of reciprocal
  relations and pairwise comparison matrices.
\newblock {\em Fuzzy Sets and Systems}, 454:74--99.

\bibitem[Boz{\'o}ki, 2017]{Bozoki2017}
Boz{\'o}ki, S. (2017).
\newblock Two short proofs regarding the logarithmic least squares optimality
  in {C}hen, {K}., {K}ou, {G}., {T}arn, {J}.~{M}., {S}ong, {Y}. (2015):
  Bridging the gap between missing and inconsistent values in eliciting
  preference from pairwise comparison matrices, {A}nnals of {O}perations
  {R}esearch 235(1):155--175.
\newblock {\em Annals of Operations Research}, 253(1):707--708.

\bibitem[Boz{\'o}ki et~al., 2016]{BozokiCsatoTemesi2016}
Boz{\'o}ki, S., Csat{\'o}, L., and Temesi, J. (2016).
\newblock An application of incomplete pairwise comparison matrices for ranking
  top tennis players.
\newblock {\em European Journal of Operational Research}, 248(1):211--218.

\bibitem[Boz\'oki et~al., 2010]{BozokiFulopRonyai2010}
Boz\'oki, S., F\"ul\"op, J., and R\'onyai, L. (2010).
\newblock On optimal completion of incomplete pairwise comparison matrices.
\newblock {\em Mathematical and Computer Modelling}, 52(1-2):318--333.

\bibitem[Boz{\'o}ki and Rapcs{\'a}k, 2008]{BozokiRapcsak2008}
Boz{\'o}ki, S. and Rapcs{\'a}k, T. (2008).
\newblock On {S}aaty's and {K}oczkodaj's inconsistencies of pairwise comparison
  matrices.
\newblock {\em Journal of Global Optimization}, 42(2):157--175.

\bibitem[Boz{\'o}ki and Tsyganok, 2019]{BozokiTsyganok2019}
Boz{\'o}ki, S. and Tsyganok, V. (2019).
\newblock The (logarithmic) least squares optimality of the arithmetic
  (geometric) mean of weight vectors calculated from all spanning trees for
  incomplete additive (multiplicative) pairwise comparison matrices.
\newblock {\em International Journal of General Systems}, 48(4):362--381.

\bibitem[Brunelli, 2018]{Brunelli2018}
Brunelli, M. (2018).
\newblock A survey of inconsistency indices for pairwise comparisons.
\newblock {\em International Journal of General Systems}, 47(8):751--771.

\bibitem[Brunelli and Fedrizzi, 2023]{BrunelliFedrizzi2023}
Brunelli, M. and Fedrizzi, M. (2023).
\newblock Inconsistency indices for pairwise comparisons and the {P}areto
  dominance principle.
\newblock {\em European Journal of Operational Research}, in press.
\newblock {DOI}:
  \href{https://doi.org/10.1016/j.ejor.2023.06.033}{10.1016/j.ejor.2023.06.033}.

\bibitem[Cavallo, 2020]{Cavallo2020}
Cavallo, B. (2020).
\newblock Functional relations and {S}pearman correlation between consistency
  indices.
\newblock {\em Journal of the Operational Research Society}, 71(2):301--311.

\bibitem[Chao et~al., 2018]{ChaoKouLiPeng2018}
Chao, X., Kou, G., Li, T., and Peng, Y. (2018).
\newblock Jie {K}e versus {A}lpha{G}o: A ranking approach using decision making
  method for large-scale data with incomplete information.
\newblock {\em European Journal of Operational Research}, 265(1):239--247.

\bibitem[Chen et~al., 2015]{ChenKouMichaelTarnSong2015}
Chen, K., Kou, G., Michael~Tarn, J., and Song, Y. (2015).
\newblock Bridging the gap between missing and inconsistent values in eliciting
  preference from pairwise comparison matrices.
\newblock {\em Annals of Operations Research}, 235(1):155--175.

\bibitem[Choo and Wedley, 2004]{ChooWedley2004}
Choo, E.~U. and Wedley, W.~C. (2004).
\newblock A common framework for deriving preference values from pairwise
  comparison matrices.
\newblock {\em Computers \& Operations Research}, 31(6):893--908.

\bibitem[Crawford and Williams, 1985]{CrawfordWilliams1985}
Crawford, G. and Williams, C. (1985).
\newblock A note on the analysis of subjective judgment matrices.
\newblock {\em Journal of Mathematical Psychology}, 29(4):387--405.

\bibitem[Csat\'o, 2013]{Csato2013a}
Csat\'o, L. (2013).
\newblock Ranking by pairwise comparisons for {S}wiss-system tournaments.
\newblock {\em Central European Journal of Operations Research},
  21(4):783--803.

\bibitem[Csat\'o, 2018]{Csato2018a}
Csat\'o, L. (2018).
\newblock Characterization of an inconsistency ranking for pairwise comparison
  matrices.
\newblock {\em Annals of Operations Research}, 261(1-2):155--165.

\bibitem[Csat\'o, 2019]{Csato2019b}
Csat\'o, L. (2019).
\newblock Axiomatizations of inconsistency indices for triads.
\newblock {\em Annals of Operations Research}, 280(1-2):99--110.

\bibitem[Driessen, 1988]{Driessen1988}
Driessen, T.~S.~H. (1988).
\newblock {\em Cooperative Games, Solutions and Applications}.
\newblock Springer Science \& Business Media, Dordrecht.

\bibitem[Duszak and Koczkodaj, 1994]{DuszakKoczkodaj1994}
Duszak, Z. and Koczkodaj, W.~W. (1994).
\newblock Generalization of a new definition of consistency for pairwise
  comparisons.
\newblock {\em Information Processing Letters}, 52(5):273--276.

\bibitem[Ergu and Kou, 2013]{ErguKou2013}
Ergu, D. and Kou, G. (2013).
\newblock Data inconsistency and incompleteness processing model in decision
  matrix.
\newblock {\em Studies in Informatics and Control}, 22(4):359--366.

\bibitem[Faramondi et~al., 2020]{FaramondiOlivaBozoki2020}
Faramondi, L., Oliva, G., and Boz{\'o}ki, S. (2020).
\newblock Incomplete analytic hierarchy process with minimum weighted ordinal
  violations.
\newblock {\em International Journal of General Systems}, 49(6):574--601.

\bibitem[Fedrizzi and Giove, 2007]{FedrizziGiove2007}
Fedrizzi, M. and Giove, S. (2007).
\newblock Incomplete pairwise comparison and consistency optimization.
\newblock {\em European Journal of Operational Research}, 183(1):303--313.

\bibitem[Forman and Gass, 2001]{FormanGass2001}
Forman, E.~H. and Gass, S.~I. (2001).
\newblock The {A}nalytic {H}ierarchy {P}rocess---{A}n exposition.
\newblock {\em Operations Research}, 49(4):469--486.

\bibitem[Golany and Kress, 1993]{GolanyKress1993}
Golany, B. and Kress, M. (1993).
\newblock A multicriteria evaluation of methods for obtaining weights from
  ratio-scale matrices.
\newblock {\em European Journal of Operational Research}, 69(2):210--220.

\bibitem[Harker, 1987a]{Harker1987a}
Harker, P.~T. (1987a).
\newblock Alternative modes of questioning in the {A}nalytic {H}ierarchy
  {P}rocess.
\newblock {\em Mathematical Modelling}, 9(3-5):353--360.

\bibitem[Harker, 1987b]{Harker1987b}
Harker, P.~T. (1987b).
\newblock Incomplete pairwise comparisons in the {A}nalytic {H}ierarchy
  {P}rocess.
\newblock {\em Mathematical Modelling}, 9(11):837--848.

\bibitem[Ishizaka and Labib, 2011]{IshizakaLabib2011}
Ishizaka, A. and Labib, A. (2011).
\newblock Review of the main developments in the analytic hierarchy process.
\newblock {\em Expert Systems with Applications}, 38(11):14336--14345.

\bibitem[Koczkodaj, 1993]{Koczkodaj1993}
Koczkodaj, W.~W. (1993).
\newblock A new definition of consistency of pairwise comparisons.
\newblock {\em Mathematical and Computer Modelling}, 18(7):79--84.

\bibitem[Koczkodaj et~al., 1999]{KoczkodajHermanOrlowski1999}
Koczkodaj, W.~W., Herman, M.~W., and Orlowski, M. (1999).
\newblock Managing null entries in pairwise comparisons.
\newblock {\em Knowledge and Information Systems}, 1(1):119--125.

\bibitem[Koczkodaj et~al., 2015]{KoczkodajKosiekSzybowskiXu2015}
Koczkodaj, W.~W., Kosiek, M., Szybowski, J., and Xu, D. (2015).
\newblock Fast convergence of distance-based inconsistency in pairwise
  comparisons.
\newblock {\em Fundamenta Informaticae}, 137(3):355--367.

\bibitem[Koczkodaj and Szybowski, 2016]{KoczkodajSzybowski2016}
Koczkodaj, W.~W. and Szybowski, J. (2016).
\newblock The limit of inconsistency reduction in pairwise comparisons.
\newblock {\em International Journal of Applied Mathematics and Computer
  Science}, 26(3):721--729.

\bibitem[Kou et~al., 2016]{KouErguLinChen2016}
Kou, G., Ergu, D., Lin, C., and Chen, Y. (2016).
\newblock Pairwise comparison matrix in multiple criteria decision making.
\newblock {\em Technological and Economic Development of Economy},
  22(5):738--765.

\bibitem[Ku{\l}akowski et~al., 2022]{KulakowskiMazurekStrada2022}
Ku{\l}akowski, K., Mazurek, J., and Strada, M. (2022).
\newblock On the similarity between ranking vectors in the pairwise comparison
  method.
\newblock {\em Journal of the Operational Research Society}, 73(9):2080--2089.

\bibitem[Mazurek et~al., 2020]{MazurekPerzinaStrzalkaKowal2020}
Mazurek, J., Perzina, R., Strza{\l}ka, D., and Kowal, B. (2020).
\newblock A new step-by-step ({SBS}) algorithm for inconsistency reduction in
  pairwise comparisons.
\newblock {\em IEEE Access}, 8:135821--135828.

\bibitem[Ozdemir, 2005]{Ozdemir2005}
Ozdemir, M.~S. (2005).
\newblock Validity and inconsistency in the analytic hierarchy process.
\newblock {\em Applied Mathematics and Computation}, 161(3):707--720.

\bibitem[Pan et~al., 2014]{PanLuLiuDeng2014}
Pan, D., Lu, X., Liu, J., and Deng, Y. (2014).
\newblock A ranking procedure by incomplete pairwise comparisons using
  information entropy and {D}empster-{S}hafer evidence theory.
\newblock {\em The Scientific World Journal}, 2014:904596.

\bibitem[Pereira and Bamel, 2023]{PereiraBamel2023}
Pereira, V. and Bamel, U. (2023).
\newblock Charting the managerial and theoretical evolutionary path of {AHP}
  using thematic and systematic review: a decadal (2012--2021) study.
\newblock {\em Annals of Operations Research}, 326(2):635--651.

\bibitem[Petr{\'o}czy and Csat{\'o}, 2021]{PetroczyCsato2021}
Petr{\'o}czy, D.~G. and Csat{\'o}, L. (2021).
\newblock Revenue allocation in {F}ormula {O}ne: {A} pairwise comparison
  approach.
\newblock {\em International Journal of General Systems}, 50(3):243--261.

\bibitem[Saaty, 1977]{Saaty1977}
Saaty, T.~L. (1977).
\newblock A scaling method for priorities in hierarchical structures.
\newblock {\em Journal of Mathematical Psychology}, 15(3):234--281.

\bibitem[Saaty, 1980]{Saaty1980}
Saaty, T.~L. (1980).
\newblock {\em The {A}nalytic {H}ierarchy {P}rocess: Planning, Priority
  Setting, Resource Allocation}.
\newblock McGraw-Hill, New York.

\bibitem[Saaty, 2008]{Saaty2008}
Saaty, T.~L. (2008).
\newblock Relative measurement and its generalization in decision making. why
  pairwise comparisons are central in mathematics for the measurement of
  intangible factors. the analytic hierarchy/network process.
\newblock {\em RACSAM - Revista de la Real Academia de Ciencias Exactas,
  Fisicas y Naturales. Serie A. Matematicas}, 102(2):251--318.

\bibitem[Schmeidler, 1969]{Schmeidler1969}
Schmeidler, D. (1969).
\newblock The nucleolus of a characteristic function game.
\newblock {\em SIAM Journal on Applied Mathematics}, 17(6):1163--1170.

\bibitem[Shiraishi and Obata, 2002]{ShiraishiObata2002}
Shiraishi, S. and Obata, T. (2002).
\newblock On a maximization problem arising from a positive reciprocal matrix
  in {AHP}.
\newblock {\em Bulletin of Informatics and Cybernetics}, 34(2):91--96.

\bibitem[Shiraishi et~al., 1998]{ShiraishiObataDaigo1998}
Shiraishi, S., Obata, T., and Daigo, M. (1998).
\newblock Properties of a positive reciprocal matrix and their application to
  {AHP}.
\newblock {\em Journal of the Operations Research Society of Japan},
  41(3):404--414.

\bibitem[Siraj et~al., 2012]{SirajMikhailovKeane2012}
Siraj, S., Mikhailov, L., and Keane, J.~A. (2012).
\newblock Enumerating all spanning trees for pairwise comparisons.
\newblock {\em Computers \& Operations Research}, 39(2):191--199.

\bibitem[Tekile et~al., 2023]{TekileBrunelliFedrizzi2023}
Tekile, H.~A., Brunelli, M., and Fedrizzi, M. (2023).
\newblock A numerical comparative study of completion methods for pairwise
  comparison matrices.
\newblock {\em Operations Research Perspectives}, 10:100272.

\bibitem[Ure{\~n}a et~al.,
  2015]{UrenaChiclanaMorente-MolineraHerrera-Viedma2015}
Ure{\~n}a, R., Chiclana, F., Morente-Molinera, J.~A., and Herrera-Viedma, E.
  (2015).
\newblock Managing incomplete preference relations in decision making: a review
  and future trends.
\newblock {\em Information Sciences}, 302:14--32.

\bibitem[Vaidya and Kumar, 2006]{VaidyaKumar2006}
Vaidya, O.~S. and Kumar, S. (2006).
\newblock Analytic hierarchy process: An overview of applications.
\newblock {\em European Journal of Operational Research}, 169(1):1--29.

\bibitem[Vargas, 1990]{Vargas1990}
Vargas, L.~G. (1990).
\newblock An overview of the analytic hierarchy process and its applications.
\newblock {\em European Journal of Operational Research}, 48(1):2--8.

\bibitem[Zhou et~al., 2018]{ZhouHuDengChanIshizaka2018}
Zhou, X., Hu, Y., Deng, Y., Chan, F.~T.~S., and Ishizaka, A. (2018).
\newblock A {DEMATEL}-based completion method for incomplete pairwise
  comparison matrix in {AHP}.
\newblock {\em Annals of Operations Research}, 271(2):1045--1066.

\end{thebibliography}

\end{document}